\documentclass[12pt,leqno]{amsart}
\usepackage{amsmath, amsthm, amssymb, xspace, mathrsfs}
\usepackage[tmargin=1.4in,bmargin=1.2in,rmargin=1.4in,lmargin=1.4in]{geometry}
\usepackage[breaklinks=true]{hyperref}
\usepackage{amsmath,amscd}
\usepackage{tikz-cd}

\theoremstyle{plain}
\newtheorem{theorem}{Theorem}[section]
\newtheorem{lemma}{Lemma}[section]

\theoremstyle{definition}
\newtheorem{defin}{Definition}[section]

\newtheorem*{problem}{The corona problem}

\newcommand{\comment}[1]{}

\begin{document}
\title[On the corona problem problem]{On the corona problem for strongly pseudoconvex domains}
\author{Akaki Tikaradze}
\email{Akaki.Tikaradze@utoledo.edu}
\address{University of Toledo, Department of Mathematics \& Statistics, 
Toledo, OH 43606, USA}

\maketitle
\begin{abstract}
In this note we solve that the corona problem for strongly pseudoconvex domains
under a certain assumption on the level sets of the corona data. This result settles a question of S. Krantz \cite{K}.

\end{abstract}
\vspace*{0.2in}
Recall that given a bounded domain $\Omega\subset \mathbb{C}^n$, the corona problem
for $\Omega$ asks if the following question has an affirmative answer. Throughout $H^{\infty}(\Omega)$ denotes the algebra
of bounded holomorphic functions on $\Omega$.

\begin{problem}

Let $f_1,\cdots, f_m\in H^{\infty}(\Omega).$ Assume that there exists $\epsilon>0$ such that $\sum_j |f_j(z)|>\epsilon$ for all $z\in \Omega.$
Does there exist $g_j\in H^{\infty}(\Omega), 1\leq j\leq m$ so that $$1=\sum_{j=1}^m f_jg_j.$$
\end{problem}

The corona problem for the unit disc was settled affirmatively by Carleson in 1962 in his celebrated paper  \cite{C}.
Later on the corona problem was solved positively for many planar domains (see for example \cite{GJ}). Currently,
there is no known domain in $\mathbb{C}$ for which the corona problem has a negative answer.

 The case for high dimensional domains is  a lot more complicated: there are examples of smooth pseudoconvex domains
 for which the corona problem has a negative answer (see \cite{FS}, \cite{S}.) Meanwhile, the corona problem is still widely open for such
 basic domains as balls and polydiscs. More generally, the answer to the corona problem is unknown for any strongly pseuodoconvex
 domain in $\mathbb{C}^n, n\geq 2.$
 
 To motivate our result, note that the assumption on functions $f_1,\cdots, f_m$ in the corona problem states
 that there exists $\epsilon>0$ such that $\bigcap_i |f_i|^{-1}(-\epsilon, \epsilon)=\emptyset.$ We solve the corona problem (for a class of domains)
 assuming that the intersection of closures of $|f_i|^{-1}(-\epsilon, \epsilon)$ in $\mathbb{C}^n$
 is empty.
 
 To state our result, it will be convenient to use the following definition.
 
 \begin{defin}
 
 Let $\Omega\subset\mathbb{C}^n$ be a domain. We say that $\Omega$ is a $\bar\partial L^{\infty}$-domain if for any smooth, bounded, 
 $\bar\partial$-closed form $\omega$ on $\Omega$, there exists a smooth bounded form $\omega'$ so that $\omega=\bar\partial(\omega').$
 
 \end{defin}
 Example of $\bar\partial L^{\infty}$-domains include products of $C^2$-smooth strongly pseudoconvex domains \cite{SH},
 as well as smooth bounded pseodoconvex finite type domains \cite{R}.

\begin{theorem}\label{main}
Let $\Omega\subset \mathbb{C}^n$ be a $\bar\partial L^{\infty}$-domain. Let $f_1,\cdots, f_m\in H^{\infty}(\Omega).$
Suppose that there exists $\epsilon>0$ so that 
$$\bigcap_{j=1}^m\overline{|f_i|^{-1}(-\epsilon, \epsilon)}=\emptyset.$$
Then there exists $g_i\in H^{\infty}(\Omega)$ so that $$1=\sum_{j=1}^mf_jg_j.$$

\end{theorem}

For $m=2,$  this result was obtained by Krantz \cite{K}. In \cite{P} a similar conclusion is obtained
under a much more restrictive assumption that there exist $j\neq k$ such that 
$\overline{|f_j|^{-1}(-\epsilon, \epsilon)}\cap \overline{|f_k|^{-1}(-\epsilon, \epsilon)}\cap \partial\Omega=\emptyset.$

For the proof we need to recall the following crucial result which should be well-known to experts.

\begin{lemma}\label{key}
Let $\Omega\subset \mathbb{C}^n$ be a $\bar\partial L^{\infty}$-domain. Let $f_1,\cdots, f_m\in H^{\infty}(\Omega).$
Assume there exist $g_i\in C^{\infty}(\Omega)$ so that  $g_j, \bar{\partial}g_j$ are bounded and $$1=\sum_{j=1}^m f_jg_j.$$
Then there exists $h_j\in H^{\infty}(\Omega)$ such that $$1=\sum_{j=1}^m f_jh_j.$$

\end{lemma}

Its proof uses the usual technique of the Koszul complex of the sequence $(f_1,\cdots, f_m)$ and is identical to that of [\cite{ST}, Corollary 3] where functions $g_j$ are taken to be $\bar{f_j}/\sum_j|f_j|^2.$
We give the proof in an effort to keep this note self contained.
\begin{proof}[Proof of Lemma \ref{key}]
At first, we define the Koszul complex on the sequence $(f_1,\cdots, f_n)$ with coefficients in bounded $C^{\infty}$-differential forms
on $\Omega.$ Let 
$$V=\oplus_{i=1}^m\mathbb{C}e_i, \quad K_{j, l}=\Lambda^jV\otimes_{\mathbb{C}}C^{\infty}_{0, l}(\Omega),$$
where $C^{\infty}_{0, l}(\Omega)$ denotes the space of $(C^{\infty})$ smooth $(0, l)$-forms on $\Omega.$
Then we have the Koszul differential $b: K_{j, l}\to K_{j-1, l}$ defined by the formula
$$b((e_{i_1}\wedge\cdots\wedge e_{i_t})\otimes \omega)=\sum_{p=1}^t(-1)^{p+1}e_{i_1}\wedge\cdots\hat{e}_{i_p}\cdots\wedge e_{i_t}\otimes(f_{i_p}\omega).$$
One can easily check that $b^2=0.$
We also have the $\bar{\partial}$ operator: $\bar{\partial}: K_{j, l}\to K_{j, l+1}$ defined as follows:
$$\bar{\partial}((e_{i_1}\wedge\cdots\wedge e_{i_t})\otimes\omega)=(e_{i_1}\wedge\cdots\wedge e_{i_t})\otimes\bar{\partial}(\omega).$$
Clearly $b$ and  $\bar{\partial}$ commute. So, $Ker(\bar{\partial})\cap K_{0,l}$ is the space of $\bar{\partial}$-closed $(0, l)$ forms.
Remark that our assumption of $(f_1,\cdots, f_m)$ implies that for any bounded 
$x\in K_{j, l}$ such that $b(x)=0,$ there exists a bounded $x'\in K_{j+1, l}$ with $b(x')=x.$ Indeed, it is straightforward
to check that the following choice of $x'=\eta(x)$ works: 
$$\eta(x)=\sum_{j=1}^m e_j\wedge g_jx, \quad b(\eta(x))=x$$

To prove the lemma we show a more general statement: If $x\in K_{j, l}$ is  bounded and 
$$b(x)=\bar{\partial}(x)=0,$$
then there exists a bounded $x'\in K_{j+1, l}$ such that $$b(x')=x, \bar\partial(x')=0.$$ We proceed by downwards induction
on $l.$ The base case of $l=n+1$ obviously holds. Since $x=b(\eta(x))$, then $$0=\bar{\partial}(b(\eta(x)))=b(\bar\partial(\eta(x))).$$
So we may apply the inductive hypothesis to $\bar\partial(\eta(x))$: there exists bounded $\bar\partial$-closed $y\in K_{j+1, l+1}$
such that $b(y)=\bar\partial(\eta(x)).$ Now, since $\Omega$ is a $\bar\partial L^{\infty}$-domain, there exists bounded $z\in K_{j+1, l}$ so that $\bar\partial(z)=y.$
Hence, $$\bar\partial(b(z))=\bar\partial(\eta(x)).$$ 
Finally, put $x_1=\eta(x)-b(z).$ Then $b(x_1)=x$ and $x_1$ is a bounded $\bar\partial$-closed
form, as desired.
\end{proof}

\begin{proof}[Proof of Theorem \ref{main}]

Put $U_j=\mathbb{C}^n\setminus \overline{|f_j|^{-1}(-\epsilon, \epsilon)}, 1\leq j\leq m.$ So $U_j, 1\leq j\leq m$ is a an open cover of $\mathbb{C}^n.$
Let $\rho_j, 1\leq j\leq m$ be a smooth partition of unity corresponding to this cover. So $\sum_j \rho_j=1$ and $\text{supp}(\rho_j)\subset U_j.$
In particular $||\bar\partial(\rho_j)||_{\infty}<\infty$ for all $j.$
Then 
$$||\bar{\partial}(\frac{\rho_j}{f_j})||_{\infty}\leq \epsilon^{-1}||\bar{\partial}(\rho)||_{\infty}, \quad ||\frac{\rho_j}{f_j}||_{\infty}\leq \epsilon^{-1}||\rho_j||_{\infty}, 
\quad 1\leq j\leq m.$$
Put $$g_j=\frac{\rho_j}{f_j}, 1\leq j\leq m.$$
 Then $g_j, \bar\partial(g_j)$ are bounded for all $j$ and
$$1=\sum_j f_jg_j.$$
 Thus we are done by Lemma \ref{key}.

\end{proof}

\end{document}